\pdfoutput=1

\documentclass[11pt,a4paper]{amsart}
\usepackage{amsfonts, amsmath, amssymb, amsthm, amssymb}
\usepackage{paralist}
\usepackage{url}

\theoremstyle{plain}
\newtheorem{thm}{Theorem}[section]
\newtheorem*{thm*}{Theorem}
\newtheorem{lem}[thm]{Lemma}

\newtheorem{prob}[thm]{Problem}
\theoremstyle{definition}

\theoremstyle{remark}
\newtheorem{rem}[thm]{Remark}
\newtheoremstyle{notethm}{}{}{\itshape}{}{\bfseries}{.}{.5em}{\thmnote{#3}}
\theoremstyle{notethm}
\newtheorem*{notethm}{}

\usepackage{bbold,bbm}

\usepackage{color}
\usepackage{graphicx}
\usepackage{overpic}

\newcommand\cA{{\mathcal A}}
\newcommand\cH{{\mathcal H}}

\newcommand\RR{{\mathbbm R}}

\newcommand\sign{\operatorname{sign}}

\newcommand\refDual{\ref{thm:primal}$^*$}

\newcommand\centralized{\widehat}
\newcommand\allOnes{\mathbb{1}}
\newcommand\allZeros{\mathbb{0}}
\renewcommand\epsilon{\varepsilon}
\newcommand\rhs{{\mathbf b}}

\begin{document}

\title[Zonotopes With Large 2D-Cuts]{Zonotopes With Large 2D-Cuts}
\author[R\"orig \and Witte \and Ziegler]{Thilo R\"orig \and Nikolaus Witte \and G\"unter~M. Ziegler}
\address{Thilo R\"orig, MA 6--2, Inst.\ Mathematics, Technische Universit\"at Berlin, D-10623 Berlin, Germany}
\email{thilosch@math.tu-berlin.de}
\address{Nikolaus Witte, MA 6--2, Inst.\ Mathematics, Technische Universit\"at Berlin, D-10623 Berlin, Germany}
\email{witte@math.tu-berlin.de}
\address{G\"unter~M. Ziegler, MA 6--2, Inst.\ Mathematics, Technische Universit\"at Berlin, D-10623 Berlin, Germany}
\email{ziegler@math.tu-berlin.de}
\thanks{The authors are supported by Deutsche Forschungsgemeinschaft, 
via the DFG Research Group ``Polyhedral Surfaces'', and a Leibniz grant.}
\date{June 03, 2008}

\begin{abstract}
  There are $d$-dimensional zonotopes with $n$ zones
  for which a $2$-dimensional central section has 
  $\Omega(n^{d-1})$ vertices.
  For $d=3$ this was known, with examples provided by
  the ``Ukrainian easter eggs'' by Eppstein et al.
  Our result is asymptotically optimal for all fixed $d\ge2$.
\end{abstract}

\keywords{Zonotopes, cuts, projections, complexity, Ukrainian easter~egg}
\subjclass[2000]{
52B05, 
52B11, 
52B12, 
52C35, 
52C40}
\maketitle

\section{Introduction}

Zonotopes, the Minkowski sums of finitely many line segments,
may also be defined as the images of cubes under affine maps,
while their duals can be described as the central sections of cross polytopes.
So, asking for images of zonotopes under projections, or for
central sections of their duals  doesn't give
anything new: We get again zonotopes, resp.\ duals of zonotopes.
The combinatorics of zonotopes and their duals is well understood (see e.g.\ 
\cite[Lect.~7]{ziegler:LOP}): The face lattice of a dual zonotope
may be identified with that of a real hyperplane arrangement.

However, surprising effects arise as soon as one asks for
\emph{sections} of zonotopes, resp.\ \emph{projections} of their duals.
Such questions arise in a variety of contexts.

\begin{figure}[h!]\centering
  \includegraphics[width=.53\textwidth]{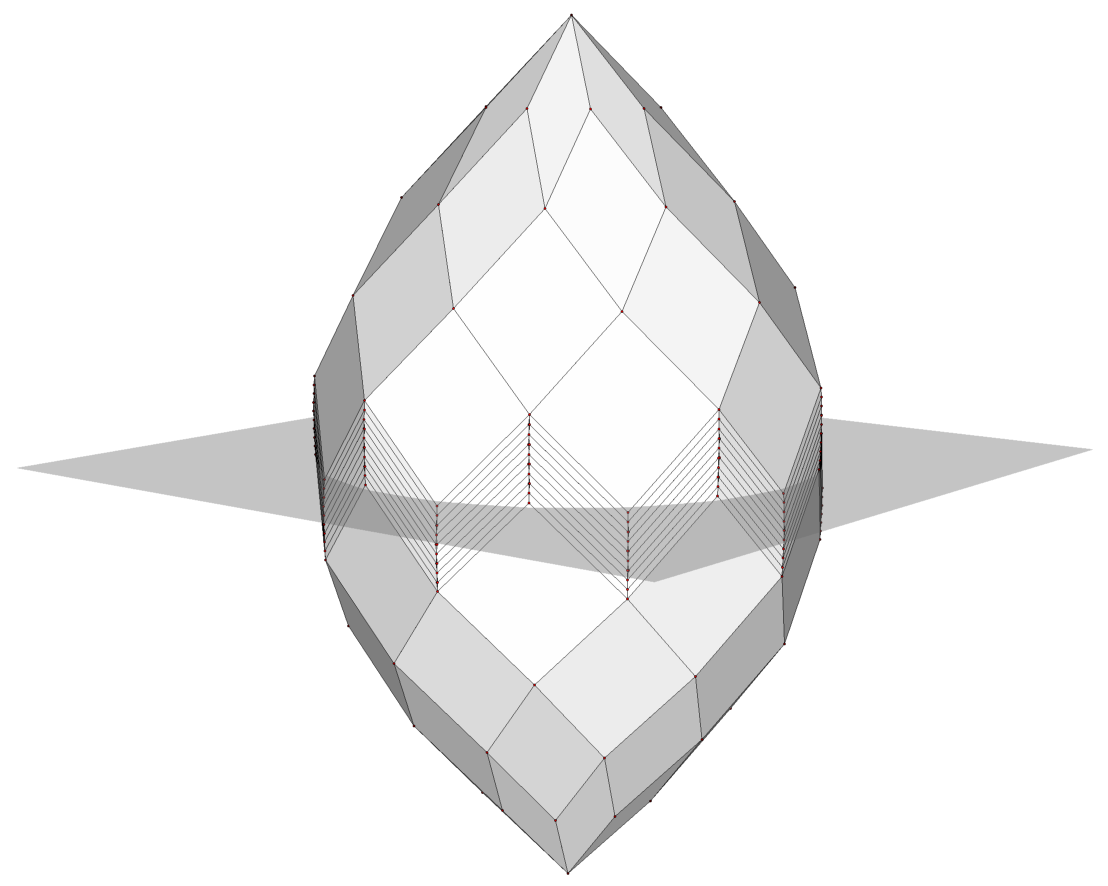}
  \includegraphics[width=.45\textwidth]{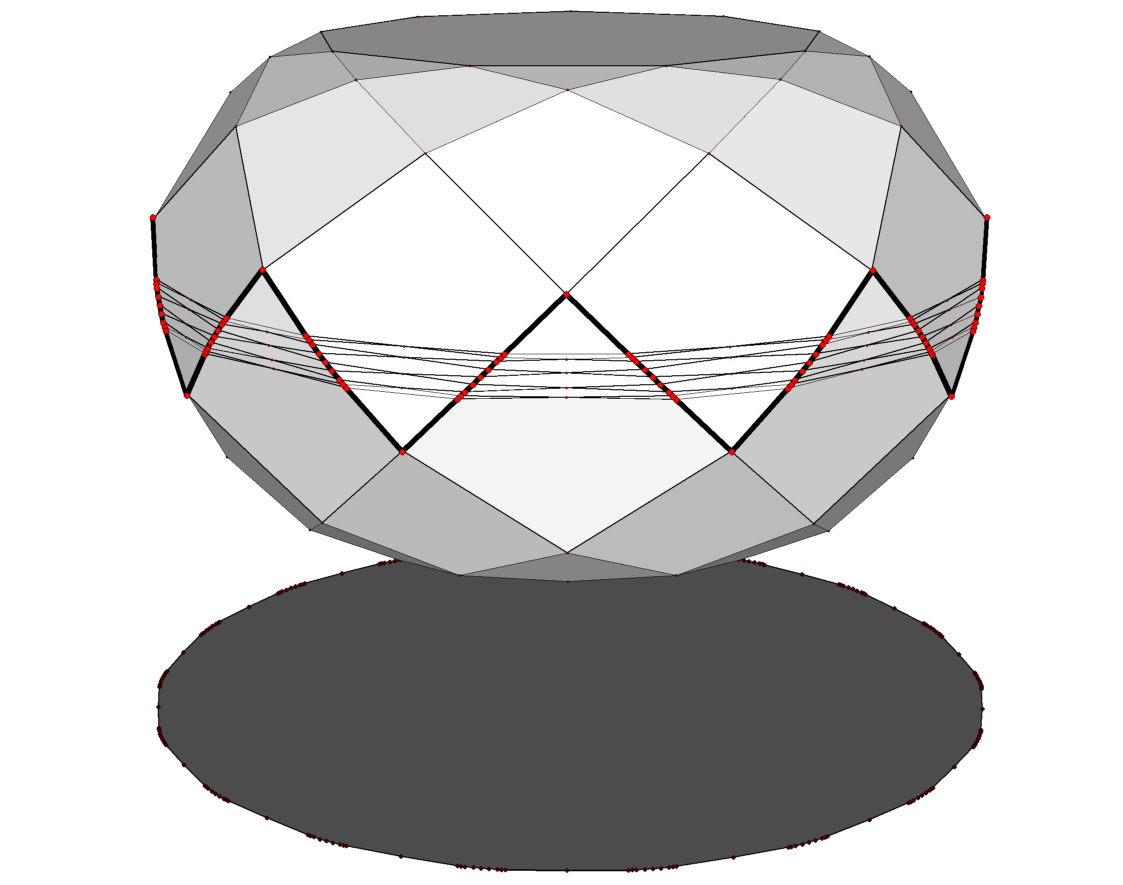}
  \caption{Eppstein's Ukrainian easter egg, and its dual. The 2D-cut, resp.\ shadow boundary, of
    size $\Omega(n^2)$ are marked.
    \label{fig:ukranian_easteregg}}
\end{figure}

For example, the ``Ukrainian Easter eggs'' as displayed by
Eppstein in his wonderful
``Geometry Junkyard'' \cite{eppstein:_ukrain_easter_egg}
are $3$-dimensional zonotopes with $n$ zones that
have a $2$-dimensional section with $\Omega(n^2)$ vertices; see also Figure~\ref{fig:ukranian_easteregg}.
For ``typical'' $3$-dimensional zonotopes with~$n$ zones one expects only
a linear number of vertices in any section, so the 
Ukrainian Easter eggs are surprising objects. Moreover,
such a zonotope has at most $2\binom n2=O(n^2)$ faces, 
so any $2$-dimensional section is a polygon with at most $O(n^2)$
edges/vertices, which shows that for dimension $d=3$ the quadratic behavior is optimal.

Eppstein's presentation of his model draws on work by Bern, 
Eppstein et al.\ \cite{BerEppGui-ESA-95}, where
also complexity questions are asked.
(Let us note that it takes a closer look to interpret the picture given
by Eppstein correctly: It is ``clipped'', and a close-up view shows that
the vertical ``chains of vertices'' hide lines of diamonds; see 
Figure~\ref{fig:easteregg_closeup}.) 
Sections of zonotopes appear also in other areas such as
Support Vector Machines and data depth; see \cite{BerEpp-WADS-01-svm},
\cite{CrispBurges}, \cite{Mosler}.
(Thanks to Marshall Bern for these references.)

\begin{figure}[h!]\centering
  \includegraphics[width=\textwidth,clip]{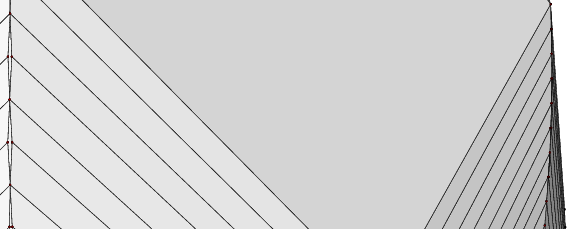}
  \caption{Close-up view of an Ukrainian Easter egg.
    \label{fig:easteregg_closeup}}
\end{figure}

It is natural to ask for high-dimensional versions of the Easter eggs.

\begin{prob}
  What is the maximal number of vertices for a $2$-dimensional central
  section of a $d$-dimensional zonotope with $n$ zones?
\end{prob}

For $d=2$ the answer is trivially $2n=\Theta(n)$, while
for $d=3$ it is of order $\Theta(n^2)$, as seen above.
We answer this question optimally for all fixed $d\ge2$.

\begin{thm}\label{thm:primal}
  For every $d\geq 2$ the maximal complexity (number of vertices)
  for a central $2$D-cut of a $d$-dimensional
  zonotope~$Z$ with~$n$ zones is~$\Theta(n^{d-1})$.
\end{thm}

The upper bound for this theorem is quite obvious: A 
$d$-dimensional zonotope with $n$ zones has at most $2\binom n{d-1}$ facets,
thus any central 2D-section has at most 
$2\binom n{d-1}=O(n^{d-1})$ edges.

To obtain lower bound constructions, it is advisable to 
look at the dual version of the problem.

\begin{prob}[Koltun {\cite[Problem~3]{OWReport}}]\label{prob:dual}
  What is the maximal number of vertices for a $2$-dimensional affine
  image (a ``\,$2$D-shadow'') of a $d$-dimensional dual zonotope with $n$ zones?
\end{prob}

Indeed, this question arose independently: It was posed
by Vladlen Koltun based on the investigation of his
``arrangement method'' for linear programming 
(\emph{see} \cite{Koltun-video}), 
which turned out to be equivalent to a Phase~I procedure for the
``usual'' simplex algorithm
(Hazan \& Megiddo \cite{hazan07:_arran_phase_one_method}).
Our construction in Section~\ref{sec:witte_ziegler} shows that
the ``shadow vertex'' pivot rule is exponential in worst-case
for the arrangement method.

Indeed,
a quick approach to Problem~\ref{prob:dual} is to use known results
about large projections of polytopes. Indeed, 
if $Z$ is a $d$-zonotope with $n$ zones, then the polar dual $Z^*$ of the zonotope $Z$
has the combinatorics
of an arrangement of $n$ hyperplanes in~$R^d$. 
The facets of~$Z^*$ are $(d-1)$-dimensional polytopes with at
most $n$ facets --- and indeed \emph{every}
$(d-1)$-dimensional polytope with at most $n$ facets arises this way.
It is known that such polytopes have exponentionally large
2D-shadows, which in the old days was bad news for the
``shadow vertex'' version of the simplex algorithm \cite{Goldfarb94}
\cite{murty80:_comput}.
Lifted to the dual $d$-zonotope $Z^*$, this also becomes relevant
for Koltun's arrangements method; 
in Section~\ref{sec:witte_ziegler} we briefly present this, and derive
the $\Omega(n^{(d-1)/2})$ lower bound.

However, what we are really heading for is an optimal result, 
dual to Theorem~\ref{thm:primal}. It will be proved in Section~\ref{sec:main}, the main part of this paper.

\begin{notethm}[Theorem~\ref{thm:primal}$^*$]
  \label{thm:dual}
  For every $d\geq 2$ the maximal complexity (number of vertices) for a $2$D-shadow of a
  $d$-dimensional zonotope~$Z^*$ with~$n$ zones is~$\Theta(n^{d-1})$.
\end{notethm}

\subsubsection*{Acknowlegements.}
We are grateful to Vladlen Koltun for his inspiration for this paper.
Our investigations were greatly helped by use of the 
\texttt{polymake} system by Gawrilow \& Joswig
\cite{polymake,gawrilow_joswig:POLYMAKE_I}. 
In particular, we have built \texttt{polymake} models that were also used
to produce the main figures in this paper.

\section{Basics}

Let $A\in\RR^{m\times d}$ be a matrix. We assume that $A$ has full (column) rank~$d$, that no
row is a multiple of another one, and none is a multiple of the first unit-vector~$(1,0,\dots,0)$.
We refer to \cite[Chap.~2]{Z10-2} or \cite[Lect.~7]{ziegler:LOP} for more detailed expositions of
real hyperplane arrangements, the associated zonotopes, and their duals.

\subsection{Hyperplane arrangements}
\label{sec:hyperplane_arrangements}

The matrix $A$ determines an essential \emph{linear hyperplane arrangement}
$\centralized{\cA}=\centralized{\cA}_A$ in~$\RR^d$, whose $m$ hyperplanes are 
    \[
    \centralized H_j = \big\{x\in\RR^d: a_j x=0\big\} \qquad \text{for }j=1,\dots,m
    \]
corresponding to the rows~$a_j$ of~$A$, and
an \emph{affine hyperplane arrangement}~$\cA=\cA_A$ in~$\RR^{d-1}$, whose
hyperplanes are
    \[
    H_j = \big\{x\in\RR^{d-1}: a_j \tbinom1x =0\big\}\qquad \text{for }j=1,\dots,m.
    \]
Given $A$, we obtain $\cA$ from $\centralized\cA$ 
by intersection with the hyperplane~$x_0 = 1$ in~$\RR^d$, a step known as \emph{dehomogenization};
similarly, we obtain $\centralized\cA$ from $\cA$ by \emph{homogenization}.

The points $x\in\RR^d$ and hence the faces of $\centralized\cA$ (and by intersection also the faces
of $\cA$) have a canonical encoding by \emph{sign vectors} 
$\sigma(x)\in\{+1,0,-1\}^m$, via the map
$s_A:x\mapsto(\sign a_1x,\dots,\sign a_mx)$. 
In the following we use the shorthand notation $\{{+},0,{-}\}$ for the set of signs.
The sign vector system $s_A(\RR^d)\subseteq\{{+},0,{-}\}^m$ 
generated this way is the \emph{oriented matroid} \cite{Z10-2} of~$\centralized\cA$.

The sign vectors $\sigma\in s_A(\RR^d)\cap\{{+},{-}\}^m$ in this system (i.e., without zeroes) 
correspond to the \emph{regions} ($d$-dimensional cells) of the arrangement~$\centralized\cA$.
For a non-empty low-dimensional cell~$F$ the sign vectors of the regions con\-tain\-ing~$F$ are
precisely those sign vectors in~$s_A(\RR^d)$ which may be obtained from~$\sigma(F)$ by replacing
each~``$0$'' by either~``$+$'' or~``$-$''.

\subsection{Zonotopes and their duals}
\label{sec:ArrangementsAndZonotopes}

The matrix $A$ also yields a zonotope $Z=Z_A$, as
\[
Z = \Big\{\sum_{i=1}^m \lambda_i a_i:\lambda_i\in[-1,+1]\text{ for }i=1,\dots,m\Big\}.
\]
(In this set-up, $Z$ lives in the vector space $(\RR^d)^*$ of row vectors, while
the dual zonotope $Z^*$ considered below consists of column vectors.)

The dual zonotope $Z^*=Z_A^*$ may be described as 
\begin{equation}\label{eq:Z*}
Z^* = \Big\{x\in\RR^d:\sum_{i=1}^m |a_i x|\le 1\Big\}.
\end{equation}
The domains of linearity of the function $f_A:\RR^d\rightarrow\RR$,
$x\mapsto \sum_{i=1}^m |a_i x|$ are the regions of the hyperplane arrangement $\centralized\cA$.
Their intersections yield the faces of~$\centralized\cA$, 
and these may be identified with the cones spanned by the proper faces of $Z^*$.
Thus the proper faces of $Z^*$ (and, by duality, the non-empty faces of~$Z$)
are identified with sign vectors in $\{{+},0,{-}\}^m$: These are the same sign vectors 
as we got for the arrangement $\centralized\cA$.

Expanding the absolute values in Equation~\eqref{eq:Z*} yields a system of~$2^m$ inequalities
describing~$Z^*$. However, a non-redundant facet description of~$Z^*$ can be obtained from~$A$ and
the combinatorics of~$\centralized{\cA}$ by considering the
inequalities~$\sigma(F)A x \leq 1$ for all sign vectors~$\sigma(F)$ of maximal cells~$F$
of~$\centralized{\cA}$:
\[
Z^* = \big\{x\in\RR^d:\sigma A x \le 1 \text{ for all } \sigma \in s_A(\RR^d)\cap \{+,-\}^m\big\}.
\]

\subsection{Projections of dual zonotopes}
\label{sec:Projections}

Let~$P$ be a $d$-polytope and let~$F\subseteq P$ be a non-empty face.
We define the \emph{matrix of normals}~$N_F$ as the matrix whose rows are the outer facet normals of all
facets containing~$F$. If~$P=\{x \in \RR^d:Nx \leq b\}$ is given by an inequality description, then~$N_F$ is the submatrix
of~$N$ formed by the rows of~$N$ that correspond
to inequalities that are tight at~$F$.
In the case when $P=Z^*$ is a dual zonotope, we derive the
following description of~$N_F$ that will be of
great use later.

\begin{lem}\label{lem:N_F}
  Let $Z^*$ be a $d$-dimensional dual zonotope corresponding to the linear arrangement~$\centralized{\cA}$
  given by the matrix~$A$, and let $F\subset Z^*$ be a non-empty face. Then the rows of~$N_F$ are the linear
  combinations~$\sigma A$ of the rows of~$A$ for all sign vectors $\sigma\in s_A(\RR^d)$ obtained
  from~$\sigma(F)$ by replacing each~``$0$'' by either~``$+$'' or~``$-$''.
\end{lem}

Let $F\subseteq P$ be a non-empty face of a $d$-polytope~$P$, and consider a projection $\pi:\RR^d\rightarrow\RR^k$.
If the outer normal vectors to the facets of $P$ that contain~$F$, projected to the
kernel of~$\pi$, positively span this kernel, then $F$ is mapped
to the face $\pi(F)$ of $\pi(P)$, which is equivalent to $F$, and $\pi^{-1}(\pi(F))\cap P=F$.
In this situation, we say that $F$ \emph{survives} the projection.

Specialized to the projection $\pi_k:\RR^d\rightarrow\RR^k$ to the first $k$ coordinates
and translated to matrix representations, this amounts to the following; see Figure~\ref{fig:survives_projection}.

\begin{lem}[see e.g.\ \cite{ziegler:PPP} \cite{Z102}]\label{lem:proj}
  Let $P$ 
  be a $d$-polytope,~$F$ a non-empty face,
  and let $N_F$ be its matrix of normals.
  If the rows of the matrix~$N_F$, truncated to the last $d-k$ components, positively
  span~$\RR^{d-k}$, then $F$ survives the orthogonal
  projection~$\pi_k$ to the first~$k$ coordinates.
\end{lem}

\begin{figure}[h!]
  \centering
  \begin{overpic}[width=.55\textwidth]{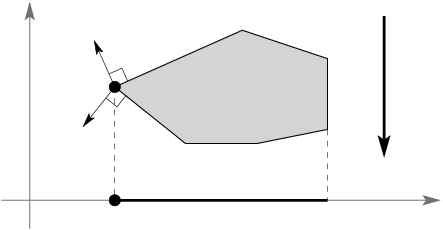}
    \put(90,32){$\pi_1$}
    \put(51,30){$P$}
    \put(44,0){$\pi_1(P)$}
    \put(19.5,30.5){$F$}
  \end{overpic}
  \caption{Survival of a face~$F$ in the projection~$\pi_1$ to the first coordinate.
    \label{fig:survives_projection}}
\end{figure}
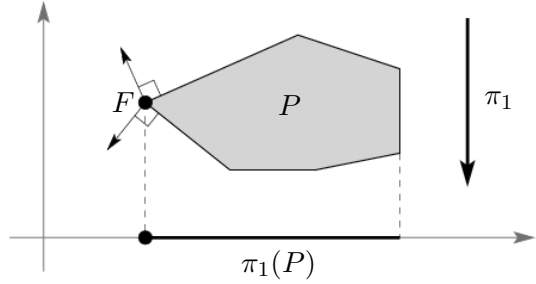

This ``projection lemma'' gives a sufficient condition for a face to survive.
In a general position situation, when proper faces of $\pi(P)$ cannot be generated
by higher-dimensional faces of~$P$, the condition of Lemma~\ref{lem:proj} is also necessary
\cite[Sect.~2.3]{Z102}.

\section{Dual Zonotopes with large 2D-Shadows}\label{sec:witte_ziegler}

In this section we present an exponential (yet not optimal) lower bound for the maximal size of
2D-shadows of dual zonotopes. It is merely a combination of known results
about polytopes and their projections.
For simplicity, we restrict to the case of odd dimension~$d$.

\begin{thm}\label{thm:witte_ziegler_dual}
  Let~$d\ge3$ be odd and~$n$ an even multiple of~$\frac{d-1}{2}$. Then
  there is a $d$-dimensional dual zonotope~$Z^\ast \subset \RR^d$
  with~$n$ zones and a projection $\pi : \RR^d \rightarrow\RR^2$ 
  such that the image~$\pi(Z^\ast)$ has at least
  $\big(\frac{2n}{d-1}\big)^\frac{d-1}{2}$ vertices.
\end{thm}

\noindent
Here is a rough sketch of the construction. 
\begin{enumerate}[(1)]
\item According to Amenta \& Ziegler~\cite[Theorem~5.2]{amenta_ziegler:DP} there are
  $(d-1)$-polytopes with~$n$ facets and exponentially many vertices such that the projection
  \mbox{$\pi_2:\RR^{d-1}\rightarrow\RR^2$} to the first two coordinates preserves all the vertices
  and thus yields a ``large'' polygon.
\item
We construct a $d$-dimensional dual zonotope~$Z^\ast$ with~$n$ zones 
that has such a $(d-1)$-polytope as a facet~$F$.
\item
The extension of $\pi_2$ to a projection
\[
\pi_3:\RR\times\RR^{d-1}\rightarrow\RR^3,\ (x_0,x)\mapsto(x_0,\pi_2(x))
\]
maps~$Z^\ast$ to a centrally symmetric $3$-polytope~$P$ with a 
\emph{large} polygon as a facet.
$P$ has a projection to~$\RR^2$ that preserves many vertices.
\end{enumerate}

\noindent
In the following we give a few details to enhance this sketch.

\subsubsection*{Some details for~{\rm(1)}:} 
Here is the exact result by Amenta \& Ziegler, which sums up
previous constructions by Goldfarb \cite{Goldfarb94} and Murty
\cite{murty80:_comput}.

\begin{thm}[{Amenta \& Ziegler~\cite{amenta_ziegler:DP}}]
\label{thm:amenta_ziegler}
Let~$d$ be odd and~$n$ an even multiple of~$\frac{d-1}{2}$. Then there  
is a $(d-1)$-polytope~$F\subset\RR^{d-1}$ with~$n$ facets
and~$\big(\frac{2n}{d-1} \big)^\frac{d-1}{2}$ vertices such that the
projection $\pi_2 : \RR^{d-1} \rightarrow \RR^2$ to the
first two coordinates preserves all vertices of~$F$. The polytope~$F$
is combinatorially equivalent to a $\big(\frac{d-1}{2}\big)$-fold product
of $\big( \frac{2n}{d-1}\big)$-gons.
\end{thm}

  Explicit matrix descriptions of deformed products of $n$-gons
  with ``large'' $4$-dimensional projections are given in~\cite{ziegler:PPP}
  \cite{Z102}. 
  These can easily be adapted (indeed, simplified) to yield explicit
  coordinates for the polytopes of Theorem~\ref{thm:amenta_ziegler}.

\subsubsection*{Some details for~{\rm(2)}:} 
We have to construct a dual zonotope~$Z^\ast$ with~$F$ as a facet.

\begin{lem} 
  \label{lem:P_as_facet} 
  Given a $(d-1)$-polytope~$F$ with~$n$ facets, there is a $d$-dimensional dual zonotope~$Z^\ast$
  with~$n$ zones that has a facet affinely equivalent to~$F$.
\end{lem}

\begin{proof}
  Let $\{x\in\RR^{d-1}:A x\le b\}$ 
  be an inequality description of~$F$,
  and let $(-b_i,A_i)$ denote the $i$-th row of the 
  matrix $(-b,A)\in\RR^{n\times d}$.

  The $n$ hyperplanes $H_i=\{x\in\RR^d:(-b_i,A_i)x=0\}$ yield a linear arrangement
  of~$n$ hyperplanes in~$\RR^d$, which may also be viewed as a fan (polyhedral complex of
  cones).  According to~\cite[Cor.~7.18]{ziegler:LOP} the fan is polytopal, and the dual $Z^\ast$ of
  the zonotope $Z$ generated by the vectors $(-b_i,A_i)$ spans the fan.  

  The resulting dual zonotope $Z^\ast$ has a facet that is \emph{projectively equivalent} to~$F$;
  however, the construction does not yet yield a facet that is affinely equivalent to~$F$.  In order
  to get this, we construct~$Z^\ast$ such that the hyperplane spanned by $F$ is $x_0 = 1$.  This is
  equivalent to constructing~$Z$ such that the vertex~$v_F$ corresponding to~$F$ is~$e_0$. Therefore
  we have to normalize the inequality description of $F$ such that
\[
\sum_{i=1}^n (-b_i,A_i) = (1,0,\ldots,0).
\]
The row vectors of~$A$ positively span~$\RR^{d-1}$ and are linearly dependent, hence there is a
linear combination of the row vectors of~$A$ with coefficients~$\lambda_i > 0$, $i = 1, \dots, n$,
which sums to~$0$. Thus if we multiply the $i$-th facet-defining inequality for $F$, corresponding
to the row vector $(-b_i,A_i)$, by
\[
  \frac{-\lambda_i}{\sum\limits_{j=1}^n \lambda_j \: b_j},
\]
  then we obtain the desired normalization of~$A$ and~$b$.
\end{proof}

\subsubsection*{Some details for~{\rm(3)}:}
The following simple lemma provides the last part of our proof;
it is illustrated in Figure~\ref{fig1}.

\begin{lem} \label{lem:proj_to_R2}
  Let~$P$ be a centrally symmetric $3$-dimensional polytope and
  let~$G\subset P$ be a $k$-gon facet. Then there exists a projection
  $\pi_G: \RR^3 \rightarrow \RR^2$ such that~$\pi_G(P)$ is a polygon
  with at least~$k$ vertices.
\end{lem}

\begin{figure}[h!]
  \centering
  \includegraphics[width=.51\textwidth]{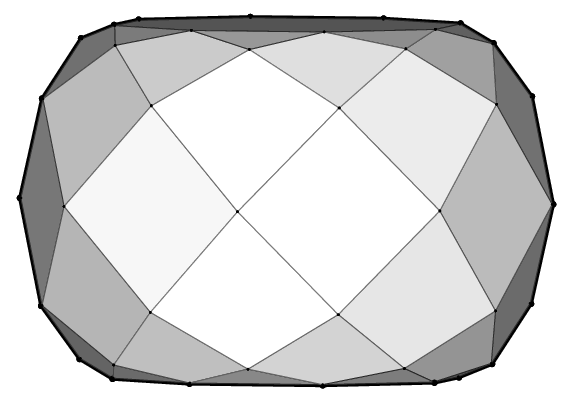}
  \hspace{1cm}
  \includegraphics[width=.39\textwidth]{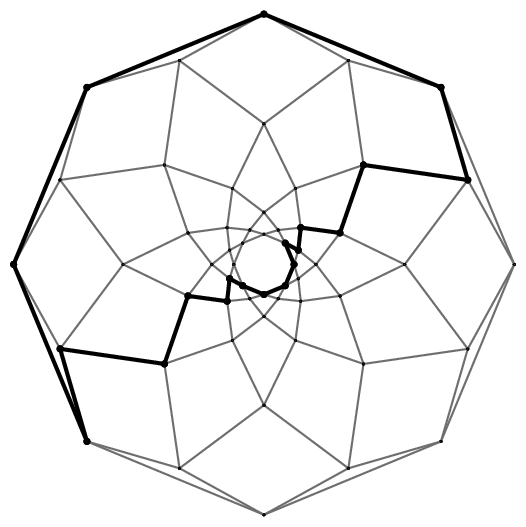}
  \caption{Shadow boundary of a centrally symmetric $3$-polytope, on the right displayed as its
    Schlegel diagram.}
  \label{fig1}
\end{figure}

\begin{proof}
  Since~$P$ is centrally symmetric, there exists a copy~$G'$ of~$G$ as a facet of~$P$ opposite and
  parallel to~$G$. Consider a projection~$\pi$ parallel to~$G$ (and to~$G'$) but otherwise generic
  and let~$n_G$ be the normal vector of the plane defining~$G$. If we perturb~$\pi$ by adding~$\pm
  \epsilon n_G$,~$\epsilon>0$, to the projection direction of~$\pi$, parts of~$\partial G$
  and~$\partial G'$ appear on the shadow boundary. Since~$P$ is centrally symmetric, the parts
  of~$\partial G$ and~$\partial G'$ appearing on the shadow boundary are the same.  Therefore
  perturbing~$\pi$ either by~$+\epsilon n_G$ or by~$-\epsilon n_G$ yields a projection~$\pi_G$ such
  that~$\pi_G(P)$ is a polygon with at least~$k$ vertices.
\end{proof}

\section{Dual Zonotopes with 2D-Shadows of Size 
$\Omega\left(n^{d-1}\right)$}
\label{sec:main}

In this section we prove our main result, Theorem~\refDual, in the
following version.

\begin{thm}\label{thm:main}
  For any $d\geq 2$ there is a $d$-dimensional dual zonotope~$Z^*$ on~$n(d-1)$ 
  zones which has a $2$D-shadow with $\Omega(n^{d-1})$ vertices.
\end{thm}

We define a dual zonotope~$Z^*$ and examine its crucial properties. These are then summarized in
Theorem~\ref{thm:main_2}, which in particular implies Theorem~\ref{thm:main}.
Figure~\ref{fig:3d-easteregg} displays a 3-dimensional example, Figure~\ref{fig:4d-easteregg} a 4-dimensional
example of our construction.

\begin{figure}[h!]\centering
  \includegraphics[width=.453\textwidth]{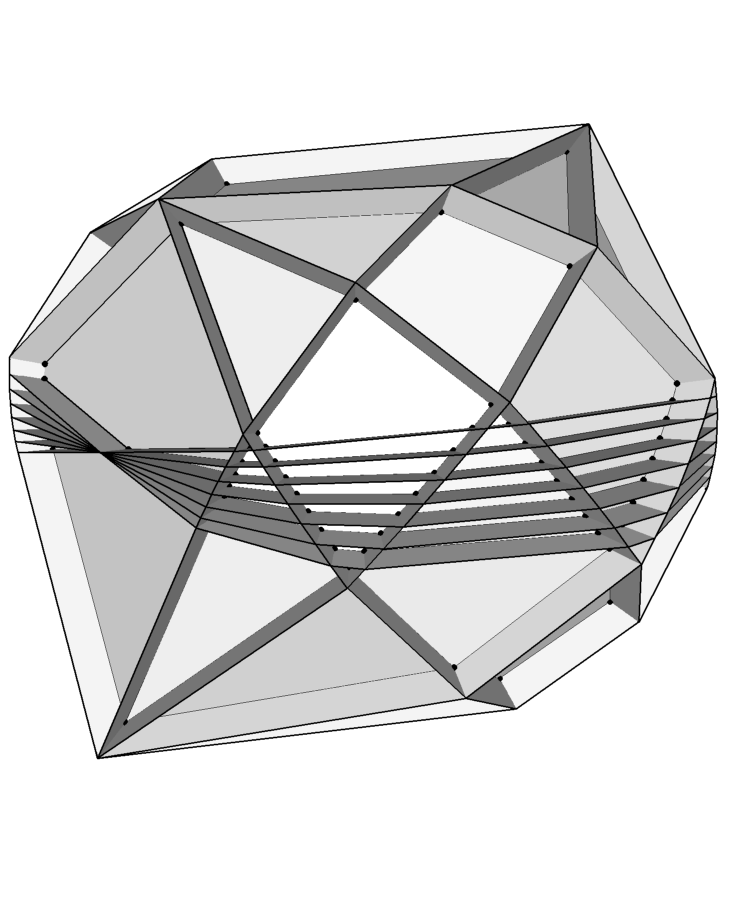}
  \hspace{1cm}
  \includegraphics[width=.447\textwidth]{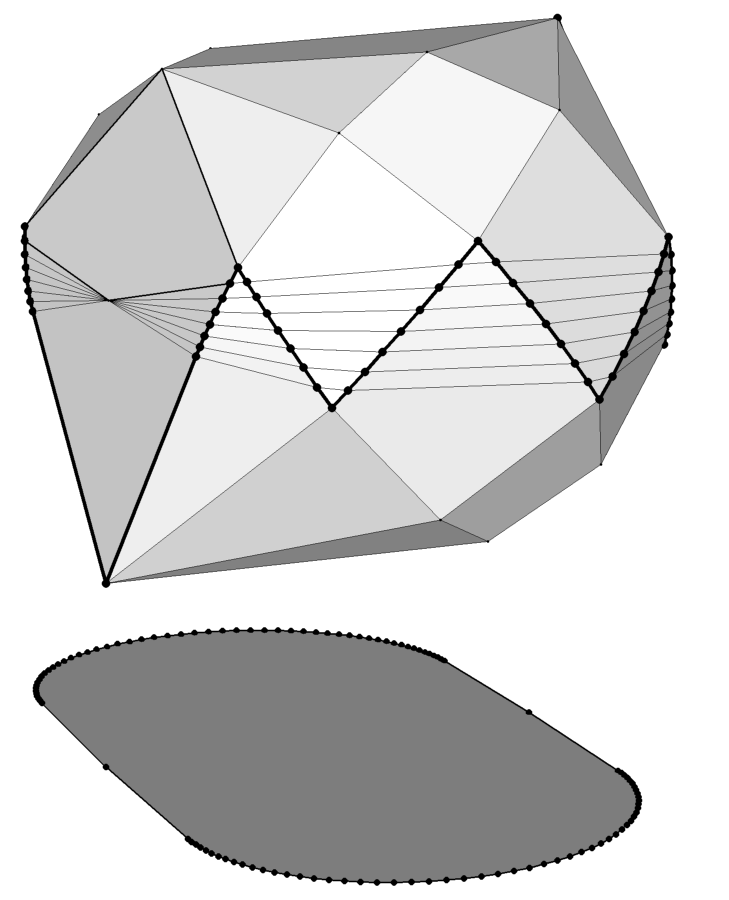}
  \caption{A dual 3-zonotope with quadratic 2D-shadow, on the left with the corresponding linear
    arrangement and on the right with its 2D-shadow.
    \label{fig:3d-easteregg}}
\end{figure}

\subsection{Geometric intuition}

Before starting with the formalism for the proof, which will be rather algebraic, here is a geometric
intuition for an inductive construction of~$Z^*=Z^*_d\subset\RR^d$, 
a $d$-dimensional zonotope on~$n(d-1)$ zones with a 2D-shadow of size~$\Omega(n^{d-1})$ when projected to the
first two coordinates.
For $d=2$ any centrally-symmetric $2n$-gon (i.e., a $2$-dimensional zonotope with $n$ zones) provides
such a dual zonotope~$Z^*_2$. The corresponding affine hyperplane arrangement~$\cA_2\subset\RR^1$
consists of~$n$ distinct points.

We derive a hyperplane arrangement~$\cA_3'\subset\RR^2$ from~$\cA_2$ by first
considering~$\cA_2\times\RR$, and then ``tilting'' the hyperplanes in~$\cA_2\times\RR$. The
hyperplanes in~$\cA_2\times\RR$ are ordered with respect to their intersections with the $x_1$-axis.
The hyperplanes in~$\cA_2\times\RR$ are tilted alternatingly in
$x_2$-direction as in Figure~\ref{fig:geometric_intuition} (left): Each black vertex of~$\cA_2$
corresponds to a north-east line and each white vertex becomes a north-west line of the arrangement
$\cA_3'$. For each vertex in the 2D-shadow of~$Z^*_2$ we obtain an edge in the 2D-shadow of the dual
3-zonotope~${Z^*_3}'$ corresponding to~$\cA_3'$.  Now~$\cA_3\subset\RR^2$ is constructed
from~$\cA_3'$ by adding a set of~$n$ parallel hyperplanes to~$\cA_3'$, all of them close
to the $x_1$-axis, and each intersecting each edge of
the 2D-shadow of~${Z^*_3}'$; see Figure~\ref{fig:geometric_intuition} (right).

\begin{figure}[h!]\centering
  \includegraphics[width=.45\textwidth]{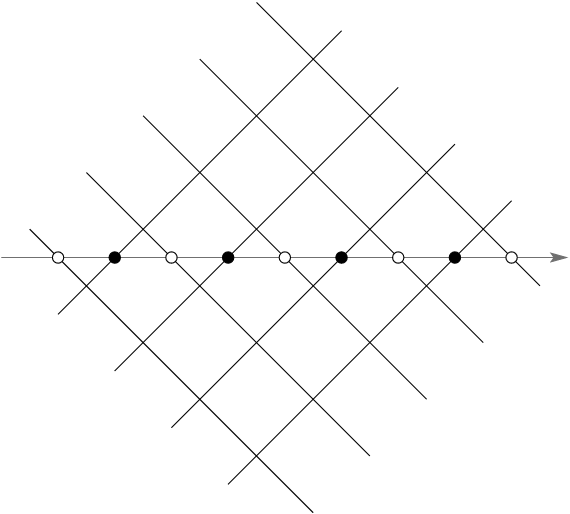}
  \hspace{1cm}
  \includegraphics[width=.45\textwidth]{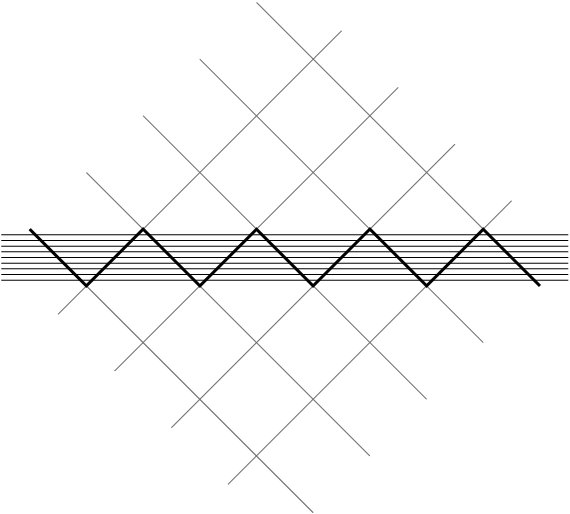}
  \caption{Constructing the arrangement~$\cA_3'$ from~$\cA_2$ (left) and~$\cA_3$ from~$\cA_3'$
    (right).
    \label{fig:geometric_intuition}}
\end{figure}

For general $d$, let~$\cH_d\subset\cA_d$ be the subarrangement of 
the~$n$ parallel hyperplanes added to~$\cA_d'$ in order to
obtain~$\cA_d$.  Then~$\cA_d'\subset\RR^{d-1 }$ is constructed from~$\cA_{d-1}\times\RR$ by tilting
the hyperplanes~$\cH_{d-1}\times\RR$, this time with respect to their intersections with the
$x_{d-2}$-axis. The corresponding $d$-dimensional dual zonotope~${Z^*_d}'$ has~$\Omega(n^{d-2})$
edges in its 2D-shadow and each of these~$\Omega(n^{d-2})$ edges is subdivided~$n$ times by the
hyperplanes in~$\cH_d$ when constructing~$\cA_d$, respectively~${Z^*_d}$. See
Figure~\ref{fig:3dArrangement} for an illustration of the arrangement~$\cA_4'$.

\begin{figure}[h!]\centering
  \includegraphics[width=\textwidth]{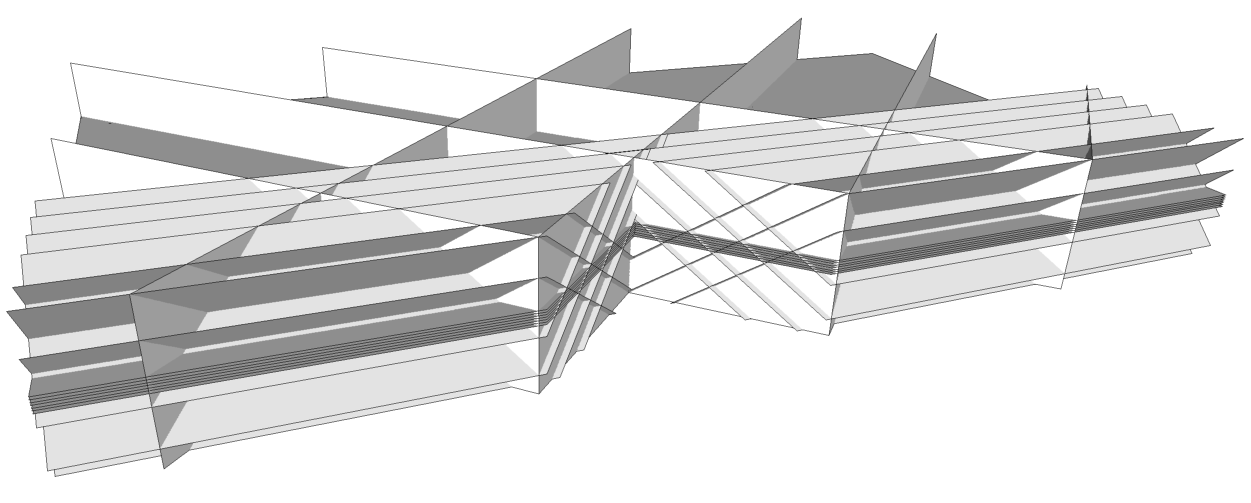}
  \caption{The affine arrangement~$\cA_4$ consists of three families of planes: 
    The first family $\cA_3'\times\RR$ forms a coarse vertical grid; 
    the second family (derived from $\cH_3\times\RR$ by tilting) forms a finer grid running from left to right;
    the last family~$\cH_4$ contains the parallel horizontal planes. \label{fig:3dArrangement}}
\end{figure}

\subsection{The algebraic construction.}

For~$k\ge1$, $n=4k+1$, and~$d\geq 2$ we define
\begin{align*}
    \rhs\phantom{'}&=(k-i)_{0\leq i\leq 2k}=\left(\begin{matrix}k\\ \vdots\\ -k\end{matrix}
    \right)\in\RR^{2k+1}\qquad\text{and}\\
    \rhs'&=\left(i-k+\frac{1}{2}\right)_{0\leq i\leq
      2k-1}=\left(\begin{matrix}-k+\tfrac{1}{2}\\ \vdots\\ k-\frac{1}{2}\end{matrix}
    \right)\in\RR^{2k}.
  \end{align*}
  Let $\allZeros,\allOnes\in\RR^\ell$ denote vectors with all entries equal to~$0$,
  respectively~$1$, of suitable size.  For convenience we index the columns of matrices from~$0$ to
  $d-1$ and the coordinates accordingly by~$x_0,\ldots,x_{d-1}$. Let 
  $\epsilon_i>0$, and for $1\le i\le d-1$ let~$A_i\in\RR^{n\times d}$ be the matrix with~$\epsilon_i
  \tbinom{\rhs\phantom{'}}{\rhs'}$ as its 0-th column
  vector, $\tbinom{\phantom{-}\allOnes}{-\allOnes}$ as its
  $i$-th column vector, $\tbinom{\allOnes}{\allOnes}$ as
  its $(i+1)$-st column vector, and zeroes otherwise. In the case~$i=d-1$ there is no $(i+1)$-st
  column of~$A_d$ and the final $\tbinom{\allOnes}{\allOnes}$-column is omitted:
\[
A_i = \!\!\!
 \bordermatrix{ &\scriptstyle0 &\scriptstyle 1 & \cdots &\scriptstyle i-1 &\scriptstyle i &\scriptstyle i+1 &\scriptstyle i+2 & \cdots &\scriptstyle d-1\cr
                &\epsilon_i\rhs\phantom{'}&\allZeros&\cdots&\allZeros&\phantom{-}\allOnes&\allOnes&\allZeros&\cdots&\allZeros\cr
                &\epsilon_i\rhs'&\allZeros&\cdots&\allZeros&-\allOnes&\allOnes&\allZeros&\cdots&\allZeros\cr}
 \in \RR^{(4k+1)\times d}.
\]

The linear arrangement~$\centralized{\cA}$ given by the $((d-1)n\times d)$-matrix~$A$ whose
horizontal blocks are the (scaled) matrices $\delta_1 A_1,\dots,\delta_{d-1} A_{d-1}$ for
$\delta_i>0$ defines a dual zonotope
by the construction of Section~\ref{sec:ArrangementsAndZonotopes}.
Since the parameters~$\delta_i$ do not change the arrangement~$\centralized{\cA}$, any choice of
the~$\delta_i$ yields the same combinatorial type of dual zonotope, but possibly different
realizations.  The choice of the~$\epsilon_i$ however may (and for sufficiently large values will)
change the combinatorics of~$\centralized{\cA}$ and hence the combinatorics of the corresponding
dual zonotope. For the purpose of constructing~$Z^*$ we set~$\alpha=\frac1{n+1}$, and
$\epsilon_i=\delta_i=\alpha^{i-1}$. This choice for $\epsilon_i$ ensures that the ``interesting''
part of the next family of hyperplanes nicely fits into the previous family. Compare
Figure~\ref{fig:geometric_intuition} (right): The interesting zig-zag part of family $A_i$ is
contained by the interval $\epsilon_i [-k -\tfrac14,k+\tfrac14]$ in $x_i$-direction 
and by $\epsilon_i [-\tfrac14,\tfrac14]$ in $x_{i+1}$-direction;
since $\epsilon_{i+1} = \tfrac{1}{n+1}\epsilon_i$ we obtain $\epsilon_{i+1}
(k+\tfrac14) < \epsilon_i \tfrac14$ and the zig-zags nicely fit into each other.
For these parameters we obtain
\begin{equation}\label{eq:A}
A = \left(
\begin{matrix}
  A_1\\ \alpha A_2\\ \vdots\\
\alpha^{d-2} A_{d-1}
\end{matrix}
\right) = \left(
\begin{array}{rrrcr}
  \rhs\phantom{'}  & \allOnes  & \allOnes &  & \\
  \rhs' & -\allOnes & \allOnes &  & \\[.9ex]
  \alpha^2 \rhs\phantom{'}  &  & \alpha \allOnes  & \alpha\allOnes & \\
  \alpha^2 \rhs' &  & -\alpha \allOnes & \alpha\allOnes & \\[.9ex]
   \vdots\quad & &  & \ddots &  \\[.9ex]
  \alpha^{2(d-2)} \rhs\phantom{'} &  &  &  & \alpha^{d-2} \allOnes\\
  \alpha^{2(d-2)} \rhs' &  &  &  & -\alpha^{d-2} \allOnes
\end{array}
\right).
\end{equation}

This matrix has size $(d-1)(4k+1)\times d= n(d-1)\times d$. The dual zonotope~$Z^*=Z_A^*$
has~$(d-1)n$ zones and is $d$-dimensional since~$A$ has rank~$d$.  According to
Section~\ref{sec:hyperplane_arrangements}, any point~$x\in\RR^d$ is labeled in
$\centralized{\cA}$ by a sign vector
$\sigma(x)=(\sigma_1,{\sigma_1}';\sigma_2,{\sigma_2}';\dots;\sigma_{d-1},{\sigma_{d-1}}')$
with~$\sigma_i\in\{+,0,-\}^{2k+1}$ and~${\sigma_i}'\in\{+,0,-\}^{2k}$. 
The following Lemma~\ref{lem:main_vertices} selects~$n^{d-1}$ vertices of~$\cA$.

\begin{lem}\label{lem:main_vertices}
  Let $H_{j_1}$,~$H_{j_2}$,~\dots,~$H_{j_{d-1}}$ be hyperplanes in~$\cA$, where each~$H_{j_i}$ is
  given by some row~$a_{j_i}$ of~$A_i$, 
  which is indexed by $j_i\in\{1,\dots,n\}$.
  Then the $d-1$ hyperplanes $H_{j_1}$,~$H_{j_2}$,~\dots,~$H_{j_{d-1}}$
  intersect in a vertex of~$\cA$ with sign vector
$(\sigma_1,{\sigma_1}';\sigma_2,{\sigma_2}';\dots;\sigma_{d-1},{\sigma_{d-1}}')
\in\{{+},0,{-}\}^{n(d-1)}$ with~$0$ at position~$j_i$
of the form
  \begin{equation}\label{eq:sign-vertices}
    (\sigma_i,{\sigma_i}') =
    \begin{cases}
      (+\dots+0-\dots-,-\dots-+\dots+) &\text{with sum $0$ or}\\
      (+\dots+-\dots-,-\dots-0+\dots+) &\text{with sum $0$}
    \end{cases}
  \end{equation}
  for each $i=1,2,\dots,d-1$. Conversely, each of these sign vectors corresponds to a vertex~$v$ of
  the arrangement. In particular,~$v$ is a generic vertex,~i.e.,~$v$ lies on exactly~$d-1$
  hyperplanes.
\end{lem}

\begin{proof}
  The intersection $v=H_{j_1}\cap H_{j_2}\cap\dots\cap H_{j_{d-1}}$ is indeed a vertex since the
  matrix minor $(a_{j_i,\ell})_{i,\ell=1,\dots, d-1}$ has full rank. We solve the system
  $A'\tbinom1v=0$ to obtain~$v$, where $A'=(a_{j_i})_{i=1,\dots, d-1}$. As we will see, the entire
  sign vector of the vertex~$v$ is determined by its~``0'' entries whose positions are given by
  the~$j_i$. Hence every sign vector agreeing with Equation~\eqref{eq:sign-vertices} determines a
  set of hyperplanes $H_{j_i}$ and thus a vertex~$v$ of the arrangement.

  To compute the position of~$v$ with respect to the other hyperplanes we take a closer look at a
  block~$A_i$ of the matrix that describes our arrangement. For an arbitrary point $x \in \RR^d$
  with $x_0=1$ we obtain
  \[
  A_i x =
  \begin{pmatrix}
    \alpha^{i-1}\rhs\phantom{'} & \phantom{-}\allOnes & \allOnes \\
    \alpha^{i-1}\rhs'& -\allOnes & \allOnes
  \end{pmatrix}
  \begin{pmatrix}
    1 \\ x_i \\ x_{i+1}
  \end{pmatrix}.
  \]
  This is equivalent to the $2$-dimensional(!) arrangement shown in Figure~\ref{fig:geometric_intuition} on the left. We
  will show that if~$x$ lies on one of the hyperplanes and if~$\left| x_{i+1} \right| <
  \tfrac{1}{4}\alpha^{i-1}$, then $x$ satisfies the required sign pattern~\eqref{eq:sign-vertices}.
  
  We start with an even simpler observation: If~$x'$ lies on one of the hyperplanes and has
  $x_{i+1}'=0$ (so in effect we are looking at a $1$-dimensional affine hyperplane arrangement), then
  there are:
  \begin{itemize}
  \item $2k$ ``positive'' row vectors~$a_j$ of~$A_i$ with~$a_j x' > 0$,
  \item $2k$ ``negative'' row vectors~$a_j$ of~$A_i$ with~$a_j x' < 0$, and
  \item one ``zero'' row vector corresponding to the hyperplane $x'$ lies on.
  \end{itemize}
  The order of the rows of~$A_i$ is such that the signs match the sign pattern
  of~$(\sigma_i,{\sigma_i}')$ in~\eqref{eq:sign-vertices}.  Since the values in~$\alpha^{i-1}\rhs$
  and~$\alpha^{i-1}\rhs'$ differ by at least~$\tfrac{1}{2}\alpha^{i-1}$ we have in fact $a_j x' \geq
  \tfrac{1}{2}\alpha^{i-1}$ for ``positive'' row vectors and $a_j x' \leq -\tfrac{1}{2}\alpha^{i-1}$
  for the ``negative'' row vectors of~$A_i$. Hence we have
  \begin{equation*}\label{eq:a_j}
    \left| a_j x' \right| \geq \tfrac{1}{2}\alpha^{i-1}.
  \end{equation*}
  If we now consider a point~$x$ with $\left| x_{i+1} \right| <
  \tfrac{1}{4}\alpha^{i-1}$ on the same hyperplane as~$x'$, then $\left| {x_{i}}' - x_{i} \right| =
  \left| x_{i+1} \right| < \tfrac{1}{4}\alpha^{i-1}$. For the row vectors $a_j$ with $a_j x' \neq 0$ 
  we obtain:
  \begin{align*}
    \left| a_j x \right| 
    &\geq \left| a_j x' \right| - \left| a_j(x-x') \right| \\
    &\geq \tfrac{1}{2}\alpha^{i-1} - (\left|x_i - {x_i}'\right| + \left|x_{i+1} - {x_{i+1}}'\right|)\\ 
    &> \tfrac{1}{2}\alpha^{i-1} - \tfrac{1}{4}\alpha^{i-1} - \tfrac{1}{4}\alpha^{i-1} = 0.
  \end{align*}
  Hence the sign pattern of $x$ is the same as the sign pattern of $x'$.
  
  We conclude the proof by showing that the required upper bound $\left| v_{i+1} \right| <
  \tfrac{1}{4}\alpha^{i-1}$ holds for the coordinates of the selected vertex~$v$. For all~$i'=1,2,\dots,d-2$ the
  inequality $a_{j_{i'}}\tbinom1v =0$ directly yields the bound $|v_{i'}| \leq
  k\alpha^{i'-1}+|v_{i'+1}|$. Further $a_{j_{d-1}}\tbinom1v = 0$ implies $|v_{d-1}| \le k
  \alpha^{d-2}$ and thus recursively
  \begin{equation*}
    \begin{aligned}\label{eq:v_i}
      |v_{i+1}| &\leq k\alpha^{i}+|v_{i+2}| 
                 \leq k\alpha^{i}+k\alpha^{i+1}+|v_{i+3}|\\
                &\leq\dots 
                 \leq k \alpha^{i}+k\alpha^{i+1}+\dots+|v_{d-1}|
                 \leq k \sum_{l=i}^{d-2}\alpha^l
                   <  k \alpha^{i} \sum_{l=0}^{\infty}\alpha^l\\[-5pt]
    & = \frac{k\alpha^{i}}{1-\alpha}
      = \frac{k}{4k+1} \alpha^{i-1}
      < \frac{1}{4}\alpha^{i-1}.\\[-8mm]
    \end{aligned}
  \end{equation*}
\end{proof}
\vspace{4mm}

\noindent
The selected vertices of Lemma~\ref{lem:main_vertices} correspond to certain vertices of the dual
zonotope~$Z^*$ associated to the arrangement~$\cA$. Rather than proving that these
vertices of~$Z^*$ survive the projection to the last two coordinates, 
we consider the edges corresponding to the
sign vectors obtained from Equation~\eqref{eq:sign-vertices} by replacing the~``0'' in
$(\sigma_{d-1},\sigma_{d-1}^{\prime})$ by either a~``$+$'' or a~``$-$'', and their negatives, which
correspond to the antipodal edges.

\begin{lem}\label{lem:main_survive}
  Let $S$ be the set of sign vectors $\pm(\sigma_1,{\sigma_1}';$
  $\sigma_2,{\sigma_2}';\dots;\sigma_{d-1},{\sigma_{d-1}}')$ of the form
  \[
  (\sigma_i,{\sigma_i}') =
  \begin{cases}
    (+\dots+0-\dots-,-\dots-+\dots+) \; &\text{with sum $0$ or}\\
    (+\dots+-\dots-,-\dots-0+\dots+) \; &\text{with sum $0$}
  \end{cases}
  \]
  for $1\le i\le d-2$ and
  \[
  (\sigma_{d-1},{\sigma_{d-1}}') = (+\dots+-\dots-,-\dots-+\dots+) \; \text{with sum $\pm1$}.
  \]
  Then the sign vectors in~$S$ correspond to $2n^{d-2}(n+1)$ edges of~$Z^*$, all of which survive the
  projection to the first two coordinates.
\end{lem}

\begin{proof}
  The sign vectors of $S$ indeed correspond to edges of~$Z^*$ since they are obtained
  from sign vectors of non-degenerate(!) vertices by substituting one~``0'' by a~``$+$'' or a~``$-$''.
  
  Further there are $2n^{d-2}(n+1)$ edges of the specified type: 
  Firstly there are~$n$ choices where to place
  the~``0'' in $(\sigma_i,{\sigma_i}')$ for each $i = 1,\dots,d-2$, which accounts for the
  factor~$n^{d-2}$. Let~$p$ be the number of~``$+$''-signs in~$\sigma_{d-1}$. Thus there are~$2k+2$
  choices for~$p$, and for each choice of~$p$ there are two choices for~${\sigma_{d-1}}'$, except
  for~$p=0$ and~$p=2k+1$ with just one choice for~${\sigma_{d-1}}'$. This amounts to $2(2k+2)-2=n+1$
  choices for~$(\sigma_{d-1},{\sigma_{d-1}}')$.  The factor of~$2$ is due to the 
  central symmetry.
  
  Let~$e$ be an edge with sign vector~$\sigma(e)\in S$. In order to apply Lemma~\ref{lem:proj} we
  need to determine the normals to the facets containing~$e$.  So let~$F$ be a facet containing~$e$.
  The sign vector~$\sigma(F)$ is obtained from~$\sigma(e)$ by replacing each~``$0$'' in~$\sigma(e)$
  by either~``$+$'' or~``$-$''; see Lemma~\ref{lem:N_F}. For brevity we encode~$F$ by a
  vector~$\tau(F)\in\{+,-\}^{d-2}$ corresponding to the choices for~``$+$'' or~``$-$'' made.
  Conversely, there is a facet~$F_\tau$ containing~$e$ for each vector $\tau\in\{+,-\}^{d-2}$,
  since~$e$ is non-degenerate.

  The supporting hyperplane for~$F$ is~$a(F) x= 1$ with $a(F)=\sigma(F) A$ being a linear
  combination of the rows of~$A$.  We compute the $i$-th component of~$a(F)$ for $i=2,3,\dots,d-1$:
  \begin{align*}
    a(F)_i &= (\sigma(F)A)_i 
    = ((\sigma_{i-1},\sigma_{i-1}')A_{i-1})_i + ((\sigma_i,\sigma_i')A_i)_i\\
    &= \alpha^{i-2}(\sigma_{i-1},\sigma_{i-1}')\left(\begin{matrix} \allOnes\\ \allOnes
      \end{matrix}\right) + \alpha^{i-1}(\sigma_{i},\sigma_{i}') \left(\begin{matrix} \phantom{-}\allOnes\\ -\allOnes
      \end{matrix}\right)
  \end{align*}
  Since we replace the zero of $(\sigma_{i-1},\sigma_{i-1}')$ by~$\tau(F)_{i-1}$ in order to
  obtain~$\sigma(F)$ from~$\sigma(e)$ we have
  $(\sigma_{i-1},\sigma_{i-1}')\tbinom{\allOnes}{\allOnes}=\tau(F)_{i-1}$.
  Since 
  $\big|(\sigma_i,\sigma_i')\tbinom{\phantom{-}\allOnes}{-\allOnes}\big|$ 
  is at most~$n$ it follows that
  \begin{itemize}
  \item $a(F)_i  \geq \phantom{-}\alpha^{i-2} - n\alpha^{i-1} = \phantom{-}\alpha^{i-1} > 0$ holds for $\tau(F)_{i-1} = +$ and
  \item $a(F)_i  \leq - \alpha^{i-2} + n\alpha^{i-1} = -\alpha^{i-1} < 0$ holds for $\tau(F)_{i-1} = -$.
  \end{itemize}
  In other words, we have for $i=2,3,\dots,d-1$:
  \begin{equation}\label{eq:a(F)}
    \sign a(F)_i = \tau(F)_i
  \end{equation}
  
  It remains to show that the last $d-2$ coordinates of the $2^{d-2}$ normals of the facets
  containing~$e$, that is, the facets $F_\tau$ for all $\tau\in\{+,-\}^{d-2}$, span $\RR^{d-2}$.
  But Equation~\eqref{eq:a(F)} implies that each of the orthants of $\RR^{d-2}$ contains one of the
  (truncated) normal vectors~$(a(F_\tau)_i)_{i=2,\dots,d-1}$.  Hence the (truncated) normals of all
  facets containing~$e$ positively span $\RR^{d-2}$ and~$e$ survives
  the projection to the first two coordinates by Lemma~\ref{lem:proj}.
\end{proof}

This completes the construction and analysis of~$Z^*$. Scrutinizing the sign vectors of the
edges specified in Lemma~\ref{lem:main_survive} one can
further show that these edges actually form a closed
polygon in~$Z^*$. Thus this closed polygon is the shadow boundary of~$Z^*$ (under projection to
the first two coordinates) and its projection is 
a $2n^{d-2}(n+1)$-gon. This yields the precise size of the projection of~$Z^*$. The reader is invited to localize the edges corresponding to the closed polygon from
Lemma~\ref{lem:main_survive} and the vertices from Lemma~\ref{lem:main_vertices} in Figures~\ref{fig:geometric_intuition} and~\ref{fig:3dArrangement}.

The following Theorem~\ref{thm:main_2} summarizes the construction of~$Z^*$ and its properties. Our
main result as stated in Theorem~\ref{thm:main} follows. Figure~\ref{fig:3d-easteregg} displays a
3-dimensional example, Figure~\ref{fig:4d-easteregg} a 4-dimensional example.

\begin{figure}[h!]\centering
  \includegraphics[width=.45\textwidth]{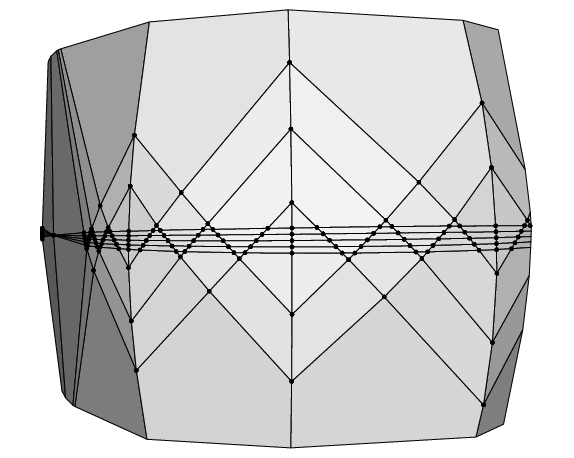}
  \hspace{1cm}
  \includegraphics[width=.45\textwidth]{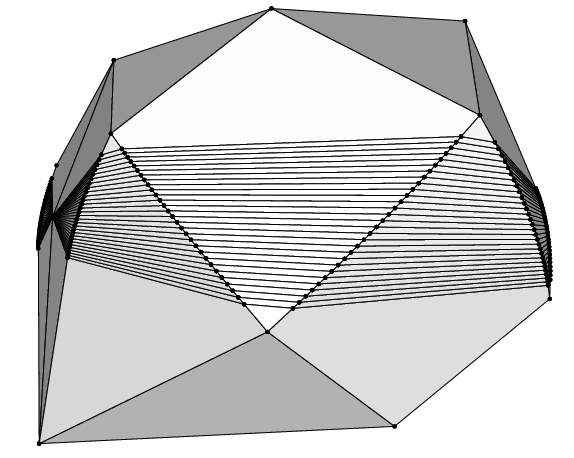}
  \caption{Two different projections of a dual 4-zonotope with cubic 2D-shadow. On the left
    the projection to the first two and last coordinate (clipped in 
    vertical direction) and on the right the projection to the first
    three coordinates. 
    \label{fig:4d-easteregg}}
\end{figure}

\begin{thm}\label{thm:main_2}
  Let~$k$ and~$d\geq 2$ be positive integers, and let~$n=4k+1$. The dual $d$-zonotope~$Z^*=Z_A^*$
  corresponding to the matrix~$A$ from Equation~\eqref{eq:A} has
  $(d-1)n$ zones and its projection to the first two coordinates has (at least) $2n^{d-1}+2n^{d-2}$ vertices.
\end{thm}

\begin{rem}
  As observed in Amenta \& Ziegler \cite[Sect.~5.2]{Z51} any result about 
  the complexity lower bound for projections to the plane ($2$D-shadows) also
  yields lower bounds for the projection to dimension $k$, a question which interpolates
  between the upper bound problems for polytopes/zonotopes ($k=d-1$) and 
  the complexity of parametric linear programming ($k=2$),
  the task to compute the LP optima for all linear combinations of two objective functions
  (see \cite[pp.~162-166]{Chv}).

  In this vein, from Theorem~\ref{thm:main} and the fact that in a dual of 
  a cubical zonotope every vertex lies in
  exactly $f_k(C_{d-1})=\binom {d-1}k2^k$ different $k$-faces (for $k<d$),
  and every such polytope contains at most $n^{d-1}$ faces of dimension~$k$,
  one derives that in the worst case 
  $\Theta(n^{d-1})$ faces of dimension $k-1$ survive in a $k$D-shadow 
  of the dual of a $d$-zonotope with $n$ zones.
\end{rem}

\bibliography{main}

\providecommand{\bysame}{\leavevmode\hbox to3em{\hrulefill}\thinspace}
\providecommand{\MR}{\relax\ifhmode\unskip\space\fi MR }
\providecommand{\MRhref}[2]{%
  \href{http://www.ams.org/mathscinet-getitem?mr=#1}{#2}
}
\providecommand{\href}[2]{#2}
\begin{thebibliography}{10}

\bibitem{Z51}
Nina Amenta and G\"unter~M. Ziegler, \emph{Shadows and slices of polytopes},
  Proceedings of the 12th Annual ACM Symposium on Computational Geometry, May
  1996, pp.~10--19.

\bibitem{amenta_ziegler:DP}
\bysame, \emph{Deformed products and maximal shadows}, Advances in Discrete and
  Computational Geometry (South Hadley, MA, 1996) (B.~Chazelle, J.~E. Goodman,
  and R.~Pollack, eds.), Contemporary Math., vol. 223, Amer. Math. Soc., 1998,
  pp.~57--90.

\bibitem{BerEpp-WADS-01-svm}
Marshall~Wayne Bern and David Eppstein, \emph{{Optimization over zonotopes and
  training support vector machines}}, Proc. 7th Worksh. Algorithms and Data
  Structures (WADS 2001) (Frank K. H.~A. Dehne, J{\"o}rg-Rudiger Sack, and
  Roberto Tamassia, eds.), Lecture Notes in Computer Science, no. 2125,
  Springer-Verlag, August 2001, pp.~111--121.

\bibitem{BerEppGui-ESA-95}
Marshall~Wayne Bern, David Eppstein, Leonidas~J. Guibas, John~E. Hershberger,
  Subhash Suri, and Jan~Dithmar Wolter, \emph{{The centroid of points with
  approximate weights}}, Proc. 3rd Eur. Symp. Algorithms (ESA 1995) (Paul~G.
  Spirakis, ed.), Lecture Notes in Computer Science, no. 979, Springer-Verlag,
  September 1995, pp.~460--472.

\bibitem{Z10-2}
Anders Bj\"orner, Michel {Las Vergnas}, Bernd Sturmfels, Neil White, and
  G\"unter~M. Ziegler, \emph{Oriented matroids}, second (paperback) ed.,
  Encyclopedia of Mathematics, vol.~46, Cambridge University Press, Cambridge,
  1999.

\bibitem{Chv}
Va\v{s}ek Chv\'{a}tal, \emph{Linear {P}rogramming}, W. H. Freeman, New York,
  1983.

\bibitem{CrispBurges}
David~J. Crisp and Christopher J.~C. Burges, \emph{A geometric interpretation
  of $\nu$-{SVM} classifiers}, NIPS (Neural Information Processing Systems)
  (S.~A. Solla, T.~K. Leen, and K.-R. M\"uller, eds.), vol.~12, MIT Press,
  Cambridge MA, 1999, pp.~244--250.

\bibitem{eppstein:_ukrain_easter_egg}
David Eppstein, \emph{Ukrainian easter egg}, in: ``The Geometry Junkyard'',
  computational and recreational geometry, 23 January 1997,\\
  \url{http://www.ics.uci.edu/~eppstein/junkyard/ukraine/}.

\bibitem{polymake}
Ewgenij Gawrilow and Michael Joswig, \emph{{\rm\tt{polymake}}, version 2.3
  (desert)}, 1997--2007, with contributions by {Thilo R{\"o}rig} and {Nikolaus
  Witte}, free software,\\ \url{http://www.math.tu-berlin.de/polymake}.

\bibitem{gawrilow_joswig:POLYMAKE_I}
Ewgenij Gawrilow and Michael Joswig, \emph{{\tt polymake}: a framework for
  analyzing convex polytopes}, Polytopes--combinatorics and computation
  (Oberwolfach, 1997), DMV Seminars, vol.~29, Birkhäuser, Basel, 2000,
  pp.~43--73.

\bibitem{Goldfarb94}
Donald Goldfarb, \emph{On the complexity of the simplex algorithm}, Advances in
  optimization and numerical analysis. Proc.\ 6th Workshop on Optimization and
  Numerical Analysis, Oaxaca, Mexico, January 1992 (Dordrecht), Kluwer, 1994,
  Based on: \emph{Worst case complexity of the shadow vertex simplex
  algorithm}, preprint, Columbia University 1983, 11~pages, pp.~25--38.

\bibitem{hazan07:_arran_phase_one_method}
Elad Hazan and Nimrod Megiddo, \emph{The ``arrangement method'' for linear
  programming is equivalent to the phase-one method}, IBM Research Report
  RJ10414 (A0708-017), IBM, August 29 2007.

\bibitem{Koltun-video}
Vladlen Koltun, \emph{The arrangement method}, Lecture at the Bay Area Discrete
  Math Day XII, April 15, 2006,\\
  \url{http://video.google.com/videoplay?docid=-6332244592098093013}.

\bibitem{Mosler}
Karl Mosler, \emph{Multivariate dispersion, central regions and depth. {T}he
  lift zonoid approach}, Lecture Notes in Statistics, vol. 165,
  Springer-Verlag, Berlin, 2002.

\bibitem{murty80:_comput}
Katta~G. Murty, \emph{Computational complexity of parametric linear
  programming}, Math.\ Programming \textbf{19} (1980), 213--219.

\bibitem{Z102}
Raman Sanyal and G\"unter~M. Ziegler, \emph{Construction and analysis of
  projected deformed products}, Preprint, October 2007, 20~pages;
  \url{http://arxiv.org/abs/0710.2162}.

\bibitem{OWReport}
Uli {Wagner (ed.)}, \emph{{Conference on Geometric and Topological
  Combinatorics: Problem Session}}, Oberwolfach Reports \textbf{4} (2006),
  no.~1, 265--267.

\bibitem{ziegler:LOP}
G\"unter~M. Ziegler, \emph{Lectures on {P}olytopes}, Graduate Texts in Math.,
  vol. 152, Springer, 1995, Revised 7th printing 2007.

\bibitem{ziegler:PPP}
\bysame, \emph{Projected products of polygons}, Electronic Research
  Announcements AMS \textbf{10} (2004), 122--134.

\end{thebibliography}

\end{document}